\documentclass[12pt]{article}
\newcommand*{\QED}{\nopagebreak\hfill\nopagebreak\rule{2mm}{2mm}\par\bigskip}%

\usepackage{amsmath,amssymb}
\usepackage{subfigure}
\usepackage{graphicx}
\usepackage{tikz}
\usetikzlibrary{automata,arrows,calc,positioning}
\usetikzlibrary{ decorations.pathreplacing}
\usepackage{mathptmx}      
\newtheorem{theorem}{Theorem}[section]
\newtheorem{lemma}[theorem]{Lemma}

\newtheorem{conjecture}[theorem]{Conjecture}
\newtheorem{proposition}[theorem]{Proposition}
\newtheorem{corollary}[theorem]{Corollary}

\newtheorem{question}[theorem]{Question}

\newenvironment{proof}{\paragraph{\textbf{Proof}}}
 {\nopagebreak\hfill\nopagebreak\rule{2mm}{2mm}\par\bigskip}

\newcommand{\maxmod}[1]{\textnormal{maxmod}(#1)} 
\newcommand{\maxmodt}[1]{\textnormal{maxmod}_{T}(#1)} 
\newcommand{\maxmodf}[1]{\textnormal{maxmod}_{F}(#1)} 

\begin{document}




\title{Maximum Modulus of Independence Roots of Graphs and Trees}
\author{Jason I.~Brown\footnote{Corresponding author. {\tt jason.brown@dal.ca}} and Ben Cameron\\
Department of Mathematics and Statistics \\
Dalhousie University, Halifax, NS B3H 3J5, Canada\\
\date{}
}
\maketitle

\begin{abstract}
 The \textit{independence polynomial} of a graph is the generating polynomial for the number of independent sets of each size and its roots are called \textit{independence roots}. We bound the maximum modulus, $\maxmod{n}$, of an independence root over all graphs on $n$ vertices and the maximum modulus, $\maxmodt{n}$, of an independence root over all trees on $n$ vertices in terms of $n$. In particular, we show that
 $$\frac{\log_3(\maxmod{n})}{n}=\frac{1}{3}+o(1)$$ 
 and
 $$\frac{\log_2(\maxmodt{n})}{n}=\frac{1}{2}+o(1).$$ 
\end{abstract}

\section{Introduction}\label{Secintro}
The \textit{independence number} of a graph $G$, denoted $\alpha(G)$, is the maximum size of an independent of $G$. The \textit{independence polynomial} of $G$, denoted $i(G,x)$, is the generating polynomial for the number of independence sets of each size:
$$i(G,x)=\sum_{k=0}^{\alpha(G)}i_kx^k,$$
\noindent where $i_k$ denotes the number of independent sets of size $k$ in $G$. (When dealing with the independence polynomials of multiple graphs, we will distinguish the coefficients with a superscript to avoid confusion, so that $i_{k}^{G}$ is the number of independent sets of size $k$ in $G$.) The roots of $i(G,x)$ are called the \textit{independence roots} of $G$.  

The independence polynomial was first introduced by Gutman and Harary in 1983 \cite{INDFIRST} and has been a fascinating object of study ever since (see Levit and Mandrescu's survey \cite{INDPOLY}). One topic that has generated much research is on the independence roots \cite{BrownCameron2018stability,BrownCameron2018vwc,BDN2000,INDROOTS,BrownNowakowski2001,Chudnovsky2007,Csikvari2013,Levit2008,Oboudi2018treeroots}. 

The roots of other graph polynomials have also been of interest and the nature and location in $\mathbb{C}$ of these roots can vary considerably depending on the polynomial (see \cite{Makowsky2014}). Determining bounds on the moduli of these roots is an important question. In 1992, the first author and Colbourn \cite{BrownColbourn1992} conjectured that the roots of reliability polynomials lie in the unit disk. The Brown-Colbourn conjecture stood for 12 years until it was shown to be false (although just barely) in \cite{RoyleSokal2004}. It was later shown that if $G$ is a connected graph on $n$ vertices and $q$ is a reliability root, then $|q|\le n-1$, yet the largest known reliability root has modulus approximately $1.113486$ \cite{BrownMol2017}. It is still believed that the reliability roots are bounded by some constant although the problem remains open. A polynomial that is more closely related to the independence polynomial is the edge cover polynomial and it was recently shown that its roots are bounded, in fact contained in the disk $|z|<\frac{(2+\sqrt{3})^2}{1+\sqrt{3}}$ \cite{Csikvari2011}. In contrast, the collection of all roots of independence polynomials \cite{INDROOTS}, domination polynomials \cite{BrownTufts2014}, and chromatic polynomials \cite{Sokal2004} are each dense in $\mathbb{C}$. 

Although these polynomials have roots with arbitrarily large moduli, an interesting question to ask is: for fixed $n$, how large can the modulus of a root of one of these polynomials be for a graph in $n$ vertices? Sokal \cite{Sokal2001} showed that all simple graphs on $n$ vertices have their chromatic roots contained in the disk $|z|\le 7.963907(n-1)$, so that the maximum moduli of chromatic roots grows at most linearly in $n$. The growth rate of domination roots is unknown. There has been work done on bounding the independence roots; for example, it was shown in \cite{BrownNowakowski2001} that for fixed $\alpha$, the largest modulus of an independence root of a graph with independence number $\alpha$ on $n$ vertices is $\left(\frac{n}{\alpha-1}\right)^{\alpha-1}+O(n^{\alpha-2})$. Although this bound is tight, the $O(n^{\alpha-2})$ term hides enough information to make it unclear if the maximum moduli of independence roots is a polynomial in $n$ or exponential in $n$. In this paper, we consider the problem of fixing $n$ as the number of vertices and determining the maximum modulus of an independence root over all graphs on $n$ vertices. We will show that the growth rate is indeed exponential. To that end, let $\maxmod{n}$ denote the maximum modulus of an independence root over all graphs on $n$ vertices and $\maxmodt{n}$ denote the maximum modulus of an independence root over all trees on $n$ vertices. We show that, in contrast to Sokal's linear bound for chromatic roots, $\maxmod{n}$ and $\maxmodt{n}$ are both exponential in $n$: in Section~\ref{sec:graph}, we prove that $$3^{\frac{n-r+3}{3}}\le \maxmod{n}\le 3^{\frac{n}{3}}+n-1,$$ where $1\le r\le 3$, while in Section~\ref{sec:trees}, we prove that $$2^{\frac{n-1}{2}} \le \maxmodt{n}\le 2^{\frac{n-1}{2}}+\frac{n-1}{2}$$ if $n$ is odd and $$2^{\frac{n-6}{2}}\le \maxmodt{n}\le 2^{\frac{n-2}{2}}+\frac{n}{2}$$ if $n$ is even.

We shall need some notation. The number of maximum independent sets in $G$ is denoted by $\xi(G)$. The number of maximal independent sets in $G$ is denoted $\mu(G)$. Note that $\xi(G)=i_{\alpha(G)}^G$, the leading coefficient of the independence polynomial of $G$. For $S\subseteq V(G)$, let $G-S$ be the graph obtained from $G$ by deleting all vertices of $S$ as well as their incident edges. If $S=\{v\}$, we will use the shorthand, $G-v$ to denote $G-\{v\}$.

\section{Bounds on the maximum modulus of independence roots}\label{sec:graph}
To bound the roots of independence polynomials, we will make extensive use of the classical Enestr\"{o}m-Kakeya Theorem which uses the ratios of consecutive coefficients of a given polynomial to describe an annulus in $\mathbb{C}$ that contains all its roots.  

\begin{theorem}[Enestr\"{o}m-Kakeya \cite{Enestrom,Kakeya1912}]\label{thm:EK} If $f (x) = a_0+ a_1x +\cdots + a_nx^n$ has positive real coefficients, then all complex roots of $f$ lie in the annulus $r \le |z| \le R$ where 

$$r = \min\left\lbrace\frac{a_i}{a_{i+1}}: 0 \le i \le n-1\right\rbrace \text{ and } R = \max\left\lbrace\frac{a_i}{a_{i+1}}: 0 \le i \le n-1\right\rbrace.$$
\end{theorem}

We will also need to make use of two basic results on computing the independence polynomial.

\begin{proposition}[\cite{INDFIRST}]\label{prop:deletion}
If $G$ and $H$ are graphs and $v\in V(G)$, then: 
\begin{itemize}
\item[i)] $i(G,x)=i(G-v,x)+x\cdot i(G-N[v],x)$.
\item[ii)] $i(G\cup H,x)=i(G,x)i(H,x)$.
\end{itemize}

\end{proposition}

Note that from Proposition~\ref{prop:deletion}, $\xi(G\cup H)=\xi(G)\cdot \xi(H)$. Our proofs are inductive and often require upper bounds $\xi(G)$ for all graphs on $n$ vertices, a collection of which can be found in \cite{Jou2000}.

\begin{theorem}[\cite{MoonMoser1965}]\label{thm:maxsetsgraph}
If $G$ is a graph of order $n\ge 2$, then 
$$\xi(G)\le\mu(G)\le g(n)= \left\{
\begin{array}{ll}
      3^{\frac{n}{3}} &~~~~\textnormal{if } n\equiv 0 \mod 3 \\      
      & \\
      4\cdot 3^{\frac{n-4}{3}} &~~~~\textnormal{if } n\equiv 1 \mod 3 \\      
      & \\
      2\cdot 3^{\frac{n-2}{3}}&~~~~\textnormal{if } n\equiv 2 \mod 3 \\      
\end{array}\right..$$

\end{theorem}

Note that an easy corollary of this is that for a graph on $n$ vertices, $\xi(G)\le\mu(G)\le 3^{\frac{n}{3}}$, since $\xi(K_1)=\mu(K_1)=1\le 3^{\frac{1}{3}}$, $3^{\frac{n}{3}}\ge 4\cdot 3^{\frac{n-4}{3}}$, and $3^{\frac{n}{3}}\ge 2\cdot 3^{\frac{n-2}{3}}$ for all $n\ge 1$.

\setcounter{subfigure}{0}
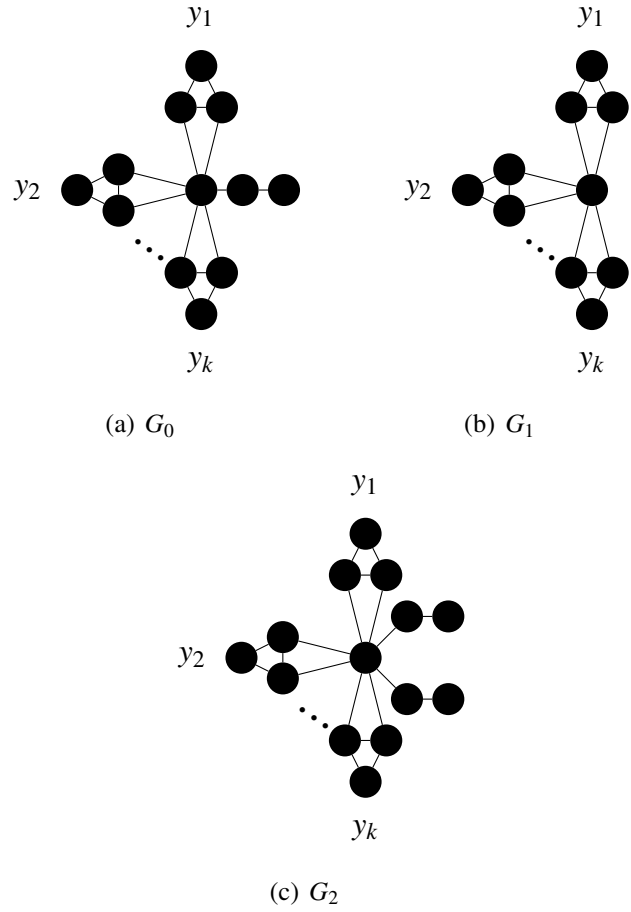
\begin{figure*}[h]
\def\c{1}
\def\r{0.55}
\centering
\subfigure[$G_0$]{
\scalebox{\c}{
\begin{tikzpicture}

\begin{scope}[every node/.style={circle,thick,draw,fill}]
    \node (11) at (1*\r,0*\r) {};
    \node (12) at (2*\r,0*\r) {};
    
   \node (10) at (0.5*\r,-2*\r) {};
    \node[label=below:$y_{k}$] (7) at (0*\r,-3*\r) {};
     \node (6) at (-0.5*\r,-2*\r) {}; 
      \node (9) at (-2*\r,0.5*\r) {};
     \node[label=left:$y_2$] (5) at (-3*\r,0*\r) {};
     \node (4) at (-2*\r,-0.5*\r) {};   
    
    \node[label=above:$y_1$](3) at (-0*\r,3*\r) {};
    \node (8) at (0.5*\r,2*\r) {};
    \node (2) at (-0.5*\r,2*\r) {};
    \node (1) at (0*\r,0*\r) {};
\end{scope}

\begin{scope}
    \path [-] (7) edge node {} (6);
    \path [-] (1) edge node {} (6);
    
    \path [-] (1) edge node {} (4);
    \path [-] (4) edge node {} (5);
    
    \path [-] (1) edge node {} (2);
    \path [-] (2) edge node {} (3);
    
    \path [-] (8) edge node {} (2);
    \path [-] (8) edge node {} (1);
    \path [-] (8) edge node {} (3);
    \path [-] (9) edge node {} (5);
    \path [-] (9) edge node {} (4);
    \path [-] (9) edge node {} (1);
    \path [-] (10) edge node {} (1);
    \path [-] (10) edge node {} (7);
    \path [-] (10) edge node {} (6);
    
    \path [-] (11) edge node {} (1);
    \path [-] (11) edge node {} (12);
\end{scope}

\path (4) -- node[auto=false]{\textbf{$\ddots$}} (6);

\end{tikzpicture}}}
\qquad
\subfigure[$G_1$]{
\scalebox{\c}{
\begin{tikzpicture}

\begin{scope}[every node/.style={circle,thick,draw,fill}]

   \node (10) at (0.5*\r,-2*\r) {};
    \node[label=below:$y_{k}$] (7) at (0*\r,-3*\r) {};
     \node (6) at (-0.5*\r,-2*\r) {}; 
      \node (9) at (-2*\r,0.5*\r) {};
     \node[label=left:$y_2$] (5) at (-3*\r,0*\r) {};
     \node (4) at (-2*\r,-0.5*\r) {};   
    
    \node[label=above:$y_1$](3) at (-0*\r,3*\r) {};
    \node (8) at (0.5*\r,2*\r) {};
    \node (2) at (-0.5*\r,2*\r) {};
    \node (1) at (0*\r,0*\r) {};
\end{scope}

\begin{scope}
    \path [-] (7) edge node {} (6);
    \path [-] (1) edge node {} (6);
    
    \path [-] (1) edge node {} (4);
    \path [-] (4) edge node {} (5);
    
    \path [-] (1) edge node {} (2);
    \path [-] (2) edge node {} (3);
    
    \path [-] (8) edge node {} (2);
    \path [-] (8) edge node {} (1);
    \path [-] (8) edge node {} (3);
    \path [-] (9) edge node {} (5);
    \path [-] (9) edge node {} (4);
    \path [-] (9) edge node {} (1);
    \path [-] (10) edge node {} (1);
    \path [-] (10) edge node {} (7);
    \path [-] (10) edge node {} (6);

\end{scope}

\path (4) -- node[auto=false]{\textbf{$\ddots$}} (6);
\end{tikzpicture}}}
\qquad
\subfigure[$G_2$]{
\scalebox{\c}{
\begin{tikzpicture}

\begin{scope}[every node/.style={circle,thick,draw,fill}]

    \node (12) at (1*\r,1*\r) {};
    \node (11) at (2*\r,1*\r) {};
    \node (13) at (1*\r,-1*\r) {};
    \node (14) at (2*\r,-1*\r) {};

   \node (10) at (0.5*\r,-2*\r) {};
    \node[label=below:$y_{k}$] (7) at (0*\r,-3*\r) {};
     \node (6) at (-0.5*\r,-2*\r) {}; 
      \node (9) at (-2*\r,0.5*\r) {};
     \node[label=left:$y_2$] (5) at (-3*\r,0*\r) {};
     \node (4) at (-2*\r,-0.5*\r) {};   
    
    \node[label=above:$y_1$](3) at (-0*\r,3*\r) {};
    \node (8) at (0.5*\r,2*\r) {};
    \node (2) at (-0.5*\r,2*\r) {};
    \node (1) at (0*\r,0*\r) {};
\end{scope}

\begin{scope}
    \path [-] (7) edge node {} (6);
    \path [-] (1) edge node {} (6);
    
    \path [-] (1) edge node {} (4);
    \path [-] (4) edge node {} (5);
    
    \path [-] (1) edge node {} (2);
    \path [-] (2) edge node {} (3);
    
    \path [-] (8) edge node {} (2);
    \path [-] (8) edge node {} (1);
    \path [-] (8) edge node {} (3);
    \path [-] (9) edge node {} (5);
    \path [-] (9) edge node {} (4);
    \path [-] (9) edge node {} (1);
    \path [-] (10) edge node {} (1);
    \path [-] (10) edge node {} (7);
    \path [-] (10) edge node {} (6);
    
    \path [-] (1) edge node {} (12);
    \path [-] (1) edge node {} (13);
    \path [-] (11) edge node {} (12);
    \path [-] (13) edge node {} (14);

\end{scope}

\path (4) -- node[auto=false]{\textbf{$\ddots$}} (6);
\end{tikzpicture}}}
\caption{Graphs with independence roots of large moduli.}%
\label{fig:conjecturedgraphs}%
\end{figure*}

\begin{proposition}\label{prop:graphroots}
For all $n\ge 1$,
\[ \maxmod{n}\ge \left\{
\begin{array}{ll}
      3^{\frac{n-3}{3}} &~~~~\textnormal{if } n\equiv 0 \mod 3 \\
      & \\
      3^{\frac{n-1}{3}} &~~~~\textnormal{if } n\equiv 1 \mod 3 \\
      &\\
      3^{\frac{n-2}{3}} &~~~~\textnormal{if } n\equiv 2 \mod 3 \\
\end{array} 
\right.. \]
\end{proposition}
\begin{proof}
The proof is in three cases depending on $n\ \textnormal{mod}\ 3$. Each relies on independence polynomials of the graphs $G_0,G_1$, and $G_2$, respectively, in Figure~\ref{fig:conjecturedgraphs} where $G_1$ is obtained by joining a central vertex to all but one vertex in each of $k$ copies of $K_3$, $G_0$ is obtained by joining one vertex in $K_2$ to the central vertex in $G_1$, and $G_2$ is obtained by joining one vertex in another copy of $K_2$ to the central vertex in $G_0$. Note that the orders of $G_0$, $G_1$, and $G_2$ are congruent to $0$, $1$, and $2$, respectively, $\textnormal{mod } 3$. We then use the Intermediate Value Theorem (IVT) to find that each has a real root of large modulus. 

From Proposition~\ref{prop:deletion}, it easily follows that
\begin{align*}
i(G_0,x)&=(1+3x)^k(1+2x)+x(1+x)^{k+1}\\
i(G_1,x)&=(1+3x)^k+x(1+x)^{k}\\
i(G_2,x)&=(1+3x)^k(1+2x)^2+x(1+x)^{k+2}.
\end{align*}

It is now straightforward to determine that 
\begin{align*}
\text{sign}\left(\lim_{x\to-\infty}i(G_0,x)\right)&=(-1)^{k}\\
\text{sign}\left(\lim_{x\to-\infty}i(G_1,x)\right)&=(-1)^{k+1}\\
\text{sign}\left(\lim_{x\to-\infty}i(G_2,x)\right)&=(-1)^{k+1}.
\end{align*}

We now prove the lower bounds for $\maxmod{n}$ by exhibiting, in each one of the cases, a graph with a real independence root with modulus larger than the bound.\\

\noindent\textbf{Case 0:} $n\equiv 0\mod 3$\\

If $n=3$, then we can use the quadratic formula to find that $P_3$ has a real independence root with modulus approximately $2.618> 1$. So we may assume that $n\ge 6$ and thus $k\ge 1$ for our analysis of $G_0$. For all $k\ge1$, we have that
\begin{align*}
i(G_0,-3^k)&=\left(1-3^{k+1}\right)^k\left(1-2\cdot 3^{k+1}\right)-3^k\left(1-3^{k}\right)^k\\
&=(-1)^k\left[(1-3^k)\left((3^{k+1}-1)^k-(3^{k+1}-3)^k\right)-3^k(3^{k}-1)^k\right]
\end{align*}

\noindent which has the same sign as $(-1)^{k+1}$ since $0>(3^{k+1}-1)^k-(3^{k+1}-3)^k.$ Thus, $i(G_0,x)$ alternates sign on $(-\infty,-3^k]$ and by the IVT and since $k=\frac{n-3}{3}$, $i(G_0,x)$ has a root in the interval $(-\infty,-3^{\frac{n-3}{3}})$.\\

\noindent\textbf{Case 1:} $n\equiv 1\mod 3$\\

If $n=1$, then $K_1$ is the only graph to consider and the result clearly holds. So we may assume that $n\ge 4$ and therefore $k\ge 1$ for our analysis of $G_1$. Since $i(G_1,-3^k)=(-1)^k((3^{k+1}-1)^k-(3^{k+1}-3)^k)$, it follows that $i(G_1,-3^k)$ has the same sign as $(-1)^k$. Thus $i(G_1,x)$ alternates sign on $(-\infty,-3^k]$ and by IVT it must have a root in the interval $(-\infty,-3^{\frac{n-1}{3}})$.\\

\noindent\textbf{Case 2:} $n\equiv 2\mod 3$\\

If $n=2$, then the graph $\overline{K_2}$ has and $-1$ as an independence root and $|-1|=1=3^{0}$. If $n=5$, then $P_5$ has a real independence root of modulus approximately $5.0489173$ which is greater than $3$. So we may assume $n\ge 8$ and therefore $k\ge 1$ for the our analysis of the graph $G_2$. We now have,

\begin{align*}
i(G_2,-3^k)&=(-1)^k\left[(1-2\cdot 3^{k})^2(3^{k+1}-1)^k-3^k(1-3^k)^2(3^k-1)^k\right]\\
&=(-1)^k\left[(1-4\cdot 3^{k}+4\cdot 3^{2k})(3^{k+1}-1)^k-(1-3^k)^2(3^{k+1}-3)^k\right]\\
&=(-1)^k\left[(1-4\cdot 3^{k}+3^{2k}+3^{2k+1})(3^{k+1}-1)^k-(1-3^k)^2(3^{k+1}-3)^k\right]\\
&=(-1)^k\left[(1-3^k)^2\left((3^{k+1}-1)^k-(3^{k+1}-3)^k\right)+\right.\\
 & \left. ~~~~~~~~~~~~~~~~(3^{2k+1}-2\cdot 3^{k})(3^{k+1}-1)^k\right]  
\end{align*}
\noindent which has sign $(-1)^k$ since $(3^{k+1}-1)^k-(3^{k+1}-3)^k>0$, $(1-3^k)^2>0$, and $(3^{2k+1}-2\cdot 3^{k})(3^{k+1}-1)^k>0$.

Therefore, IVT gives that $i(G_2,x)$ must have a root in the interval $(-\infty,-3^{\frac{n-2}{3}})$. This completes the proof.
\end{proof}

Therefore, $\maxmod{n}$ is at least exponential in $n$. We require the next two lemmas to put an upper bound on $\maxmod{n}$.

\begin{lemma}\label{lem:indnum}
For all graphs $G$ with at least one edge, there exists a non-isolated vertex $v$ such that $\alpha(G)=\alpha(G-v)\ge \alpha(G-N[v])+1$.
\end{lemma}
\begin{proof}
Let $G$ be a graph with at least one edge. It is clear that for any vertex $v$ of $G$, $\alpha(G)\ge \alpha(G-N[v])+1$, since any maximum independent set in $G-N[v]$ will still be independent in $G$ with the addition of $v$. Suppose that for all vertices $v\in V(G)$, that $\alpha(G)>\alpha(G-v)$. Then every vertex belongs to every maximum independent set. However, $G$ has at least one edge, so the vertices incident with this edge cannot belong to the same independent set, which contradicts both of these vertices being in every maximum independent set. Therefore, there exists some $v\in V(G)$ incident with some edge such that $$\alpha(G)=\alpha(G-v)\ge \alpha(G-N[v])+1.$$
\end{proof}

\begin{lemma}\label{lem:1maxindset}
If $G$ is a graph on $n$ vertices such that $\xi(G)=1$, then $i_{\alpha(G)-1}\le 3^{\frac{n}{3}}+n-1$.
\end{lemma}
\begin{proof}
Let $G$ be a graph on $n$ vertices such that $\xi(G)=1$. Every independent set of size $\alpha(G)-1$ is either maximal or is a subset of the one independent set of size $\alpha(G)$. Therefore, $i_{\alpha(G)-1}\le \mu(G)-1+\alpha(G)\le 3^{\frac{n}{3}}+n-1$ (subtracting $1$ from $\mu(G)$ to account for the one maximum independent set) by the note following Theorem~\ref{thm:maxsetsgraph}.
\end{proof}

\begin{theorem}\label{thm:graphrootbound}
For all $n\ge 1$, $\maxmod{n}\le 3^{\frac{n}{3}}+n-1$. 
\end{theorem}
\begin{proof}
We actually prove the stronger result that for a graph on $n$ vertices, the ratios of coefficients, given by $R$ in the statement of Theorem~\ref{thm:EK} (the Enestr\"{o}m-Kakeya Theorem), of its independence polynomial are bounded above by $3^{\frac{n}{3}}+n-1$. It then follows directly from the Enestr\"{o}m-Kakeya Theorem that the roots are bounded by this value. We proceed by induction on $n$. 

The results hold for graphs on $n\le 5$ vertices by straightforward checking the ratios of consecutive coefficients of the independence polynomials of all $52$ graphs in Maple. Now suppose the result holds for all $3\le k<n$, and let $G$ be a graph on $n$ vertices. If $G$ has no edges, then we are done, since $G$ has only $-1$ as an independence root in this case. Therefore, suppose $G$ has at least one edge. Let $v$ be a nonisolated vertex in $G$ such that $\alpha(G)=\alpha(G-v)\ge \alpha(G-N[v])+1$, noting that $v$ exists by Lemma~\ref{lem:indnum}.  Now, by Proposition~\ref{prop:deletion},

\begin{align*}
i(G,x)&=i(G-v,x)+x\cdot i(G-N[v],x)\\
&=\sum_{k=0}^{\alpha(G-v)}i_{k}^{G-v}x^k+x\sum_{k=0}^{\alpha(G-N[v])}i_{k}^{G-N[v]}x^k\\
&=1+\sum_{k=1}^{\alpha(G-v)}i_{k}^{G-v}x^k+\sum_{k=1}^{\alpha(G-N[v])+1}i_{k-1}^{G-N[v]}x^{k}. \tag{1}\label{eq:1}
\end{align*}

\noindent We now have two cases.\\

\noindent\textbf{Case 1:} $\alpha(G)=\alpha(G-v)=\alpha(G-N[v])+1$.\\

In this case, (\ref{eq:1}) gives
$$i(G,x)=1+\sum_{k=1}^{\alpha(G-v)}\left(i_{k}^{G-v}+i_{k-1}^{G-N[v]}\right)x^{k}.$$

This gives the following ratios between coefficients,

$$
\frac{1}{n}\text{ and }\frac{i_{k}^{G-v}+i_{k-1}^{G-N[v]}}{i_{k+1}^{G-v}+i_{k}^{G-N[v]}}\text{ for $k=1,2,\ldots,\alpha(G-N[v])$}.
$$

For all $n\ge 1$, $\frac{1}{n}<3^{\frac{n}{3}}+n-1$, and by the inductive hypothesis,

\begin{align*}
\frac{i_{k}^{G-v}+i_{k-1}^{G-N[v]}}{i_{k+1}^{G-v}+i_{k}^{G-N[v]}}&< \frac{\left(3^{\frac{n-1}{3}}+n-2\right)i_{k+1}^{G-v}+\left(3^{\frac{n-|N[v]|}{3}}+n-1-|N[v]|\right)i_{k}^{G-N[v]}}{i_{k+1}^{G-v}+i_{k}^{G-N[v]}} \\
&\le \frac{\left(3^{\frac{n-1}{3}}+n-2\right)\left(i_{k+1}^{G-v}+i_{k}^{G-N[v]}\right)}{i_{k+1}^{G-v}+i_{k}^{G-N[v]}}\\
&=3^{\frac{n-1}{3}}+n-2\\
&<3^{\frac{n}{3}}+n-1. 
\end{align*}

\noindent\textbf{Case 2:} $\alpha(G)=\alpha(G-v)>\alpha(G-N[v])+1$.\\

In this case, the independence polynomial is obtained from (\ref{eq:1}) as,
In this case, (\ref{eq:1}) gives
$$i(G,x)=1+\sum_{k=1}^{\alpha(G-N[v])+1}\left(i_{k}^{G-v}+i_{k-1}^{G-N[v]}\right)x^{k}+\sum_{\alpha(G-N[v])+2}^{\alpha(G-v)}i_k^{G-v}x^k.$$

This gives four different forms for $\frac{i_k^G}{i_{k+1}^G}$. The first two, namely $\frac{1}{n}$ and $\frac{i_{k}^{G-v}+i_{k-1}^{G-N[v]}}{i_{k+1}^{G-v}+i_{k}^{G-N[v]}}$, are less than or equal to $3^{\frac{n}{3}}+n-1$ for each $k=1,2,\ldots,\alpha(G-N[v])$ by the same argument as Case 1. This leaves,

$$\frac{i_{\alpha(G-N[v])+1}^{G-v}+i_{\alpha(G-N[v])}^{G-N[v]}}{i_{\alpha(G-N[v])+2}^{G-v}},\text{ and } \frac{i_k^{G-v}}{i_{k+1}^{G-v}}\text{ for $k\ge \alpha(G-N[v])+2$}
$$

\noindent By the inductive hypothesis, $\frac{i_k^{G-v}}{i_{k+1}^{G-v}}\le 3^{\frac{n-1}{3}}+n-2< 3^{\frac{n}{3}}+n-1$, so we are left only with $\frac{i_{\alpha(G-N[v])+1}^{G-v}+i_{\alpha(G-N[v])}^{G-N[v]}}{i_{\alpha(G-N[v])+2}^{G-v}}$.

In this case, we first show that $|N[v]|\ge 3$. As $v$ is not isolated, $|N[v]|\ge  2$. If $|N[v]|=2$, then $v$ is a leaf, and since $v$ was chosen such that $\alpha(G)=\alpha(G-v)\ge \alpha(G-N[v])+1$, $v$ is not in every maximum independent set in $G$. But every maximum independent set in $G$ must contain either $v$ or its neighbour, so $\alpha(G-v)=\alpha(G-N[v])+1$ as covered in Case 1. Therefore, we may assume $|N[v]|\ge 3$. We also note that $$\frac{i_{\alpha(G-N[v])+1}^{G-v}+i_{\alpha(G-N[v])}^{G-N[v]}}{i_{\alpha(G-N[v])+2}^{G-v}}=\frac{i_{\alpha(G-N[v])+1}^{G-v}}{i_{\alpha(G-N[v])+2}^{G-v}}+\frac{\xi(G-N[v])}{i_{\alpha(G-N[v])+2}^{G-v}}.$$
There are three subcases to consider.\\

\noindent\textbf{Case 2a:} $\alpha(G-N[v])+2<\alpha(G-v)=\alpha(G)$.\\

If $\alpha(G-N[v])+2<\alpha(G-v)$, then $G-v$ has an independent set of size $\alpha(G-N[v])+3$. Therefore, $i_{\alpha(G-N[v])+2}^{G-v}\ge \alpha(G-N[v])+3\ge 3$, since any independent set of size  $k$, contains at least $\binom{k}{k-1}=k$ independent sets of size $k-1$. Now by the inductive hypothesis and the note following Theorem~\ref{thm:maxsetsgraph},

\begin{align*}
\frac{i_{\alpha(G-N[v])+1}^{G-v}}{i_{\alpha(G-N[v])+2}^{G-v}}+\frac{\xi(G-N[v])}{i_{\alpha(G-N[v])+2}^{G-v}}&\le 3^{\frac{n-1}{3}}+n-2+\frac{\xi(G-N[v])}{3}\\\
&\le 3^{\frac{n-1}{3}}+n-2+3^{\frac{n-|N[v]|-3}{3}}\\
&\le 3^{\frac{n-1}{3}}+n-2+3^{\frac{n-6}{3}}\\
&=3^{\frac{n}{3}}\left(3^{\frac{-1}{3}}+\frac{1}{9}\right)+n-2\\
&\le 3^{\frac{n}{3}}+n-1.
\end{align*}

\noindent\textbf{Case 2b:} $\alpha(G-N[v])+2=\alpha(G-v)=\alpha(G)$ and $|N[v]|\ge 4$.\\

In this case, by the inductive hypothesis and the note following Theorem~\ref{thm:maxsetsgraph},

\begin{align*}
\frac{i_{\alpha(G-N[v])+1}^{G-v}}{i_{\alpha(G-N[v])+2}^{G-v}}+\frac{\xi(G-N[v])}{i_{\alpha(G-N[v])+2}^{G-v}}&\le 3^{\frac{n-1}{3}}+n-2+\xi(G-N[v])\\\
&\le 3^{\frac{n-1}{3}}+n-2+3^{\frac{n-|N[v]|}{3}} \\
&\le 3^{\frac{n-1}{3}}+n-2+3^{\frac{n-4}{3}}\\
&=3^{\frac{n}{3}}\left(3^{\frac{-1}{3}}+3^{\frac{-4}{3}}\right)+n-2\\
&< 3^{\frac{n}{3}}+n-1.
\end{align*}

\noindent\textbf{Case 2c:} $\alpha(G-N[v])+2=\alpha(G-v)$ and $|N[v]|= 3$.\\

We break this final case into two subcases bases on the size of $i_{\alpha(G-N[v])+2}^{G-v}$. First, if $i_{\alpha(G-N[v])+2}^{G-v}\ge 2$, then by the inductive hypothesis and the note following Theorem~\ref{thm:maxsetsgraph},

\begin{align*}
\frac{i_{\alpha(G-N[v])+1}^{G-v}}{i_{\alpha(G-N[v])+2}^{G-v}}+\frac{\xi(G-N[v])}{i_{\alpha(G-N[v])+2}^{G-v}}&\le 3^{\frac{n-1}{3}}+n-2+\frac{\xi(G-N[v])}{2} \\\
&\le 3^{\frac{n-1}{3}}+n-2+\frac{3^{\frac{n-3}{3}}}{2}\\
&=3^{\frac{n}{3}}\left(3^{\frac{-1}{3}}+\frac{1}{3}\right)+n-2\\
&\le 3^{\frac{n}{3}}+n-1.
\end{align*}

Note if some maximum independent set in $G$ contained $v$, then this set with $v$ removed would be an independent set of size $\alpha(G)-1=\alpha(G-N[v])+1$ in $G-N[v]$, which is a contradiction. Therefore, the maximum independent sets in $G$ and $G-v$ are exactly the same sets and, in particular, $\xi(G)=\xi(G-v)$. Now, if $$1=i_{\alpha(G-N[v])+2}^{G-v}=\xi(G-v)=\xi(G),$$ then Lemma~\ref{lem:1maxindset} applied to $G$ gives a bound on $i_{\alpha(G)-1}^G$ in the last line of the following,

\begin{align*}
\frac{i_{\alpha(G-N[v])+1}^{G-v}+i_{\alpha(G-N[v])}^{G-N[v]}}{i_{\alpha(G-N[v])+2}^{G-v}}&=\frac{i_{\alpha(G)-1}^{G}}{i_{\alpha(G)}^{G}}\\
&=i_{\alpha(G)-1}^{G}\\
&\le 3^{\frac{n}{3}}+n-1.
\end{align*}

Now, if $z$ is an independence root of $G$, then, by the Enestr\"{o}m-Kakeya Theorem, $|z|\le 3^{\frac{n}{3}}+n-1$.
\end{proof}

Proposition~\ref{prop:graphroots} and Theorem~\ref{thm:graphrootbound} give the following corollary.

\begin{corollary}\label{cor:asymgraphbound}
$$\frac{\log_{3}(\maxmod{n})}{n}=\frac{1}{3}+o(1).$$ 
\QED
\end{corollary}


\section{Bounds for trees of order $n$}\label{sec:trees}

Now that we have determined bounds on $\maxmod{n}$, a natural extension of this is to determine the \textit{largest} modulus an independence root can obtain among all graphs of order $n$ in a specific family of graphs. In particular, the bound we obtained for $\maxmod{n}$ seems to be much too large when we restrict our attention to trees. In this section, we consider $\maxmodt{n}$, the maximum modulus of an independence root over all trees on $n$ vertices. 


%

Let $T_k$ be the tree obtained by gluing $k$ copies of $P_3$ together at a leaf (see Figure~\ref{fig:Tk}). This tree is known \cite{Sagan1988} to have the largest number of maximal independent sets among trees on $2k+1$ vertices and we will show that it also has the largest ratio of consecutive coefficients among all such trees, and therefore provides an upper bound on $\maxmodt{n}$.

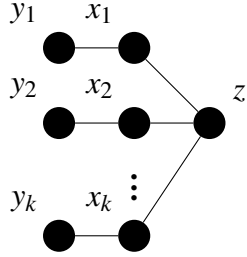
\begin{figure}[h]
\def\c{1}
\def\r{1}
\centering
\scalebox{\c}{
\begin{tikzpicture}
\begin{scope}[every node/.style={circle,thick,draw,fill}]
     \node[label=above left:$y_{k}$] (7) at (-2*\r,-1.5*\r) {};
     \node[label=above left:$x_{k}$] (6) at (-1*\r,-1.5*\r) {}; 
     \node[label=above left:$y_2$] (5) at (-2*\r,0*\r) {};
     \node[label=above left:$x_2$] (4) at (-1*\r,0*\r) {};   
    
    \node[label=above left:$y_1$](3) at (-2*\r,1*\r) {};
    \node[label=above left:$x_1$] (2) at (-1*\r,1*\r) {};
    \node[label=above right:$z$] (1) at (0*\r,0*\r) {}; 
\end{scope}

\begin{scope}
    \path [-] (7) edge node {} (6);
    \path [-] (1) edge node {} (6);
    
    \path [-] (1) edge node {} (4);
    \path [-] (4) edge node {} (5);
    
    \path [-] (1) edge node {} (2);
    \path [-] (2) edge node {} (3);
 
\end{scope}

\path (4) -- node[auto=false]{\textbf{\vdots}} (6);
\end{tikzpicture}}
\caption{The tree $T_k$ on $2k+1$ vertices that has independence root in $[-2^k-k,-2^k)$.}%
\label{fig:Tk}%
\end{figure}

\begin{theorem}[\cite{Wilf1986}]\label{thm:maxsetstree}
If $G$ is a tree of order $n\ge 2$, then

$$\xi(G)\le t'(n)= \left\{
\begin{array}{ll}
      2^{\frac{n-3}{2}} &~~~~\textnormal{if } n\textnormal{ is odd} \\
      & \\
      2^{\frac{n-2}{2}}+1 &~~~~\textnormal{if } n\textnormal{ is even}\\
\end{array}\right..$$

\end{theorem}

We need to extend our notation to $\maxmodf{n}$, the maximum modulus of an independence root of a forest of order $n$; clearly $\maxmodt{n}\le \maxmodf{n}$. The following lemma will be required in proving an upper bound on the ratio of consecutive coefficients of the independence polynomials of forests.
\begin{lemma}\label{lem:topratiobound}
If $F$ is a forest on $n$ vertices, $n\ge 2$, and $v\in V(F)$, then
$$\frac{\xi(F)}{\xi(F-v)}\le \left\{
\begin{array}{ll}
      2^{\frac{n-3}{2}}+1 &~~~~\textnormal{if }n\textnormal{ is odd} \\
      & \\
      2^{\frac{n-2}{2}}+1 &~~~~\textnormal{if }n\textnormal{ is even}\\
\end{array}\right.. $$
\end{lemma}
\begin{proof}
Let $F=H_1\cup H_2\cup\cdots\cup H_k$, where $k\ge 1$, and each $H_i$ is a connected component of $F$. Suppose, without loss of generality, that $v\in H_k$. Then $$F-v=H_1\cup H_2\cup\cdots H_{k-1}\cup F',$$ where $F'$ is the forest obtained from deleting $v$ from $H_k$ (note that if $v$ was an isolated vertex in $F$, then $F'$ may have no vertices and $\xi(H_k)=1$). Now we have,

\begin{align*}
\frac{\xi(F)}{\xi(F-v)}&=\frac{\xi(H_1)\cdot\xi(H_2)\cdots\xi(H_k)}{\xi(H_1)\cdot\xi(H_2)\cdots\xi(H_{k-1})\cdot\xi(F')}\\
&=\frac{\xi(H_k)}{\xi(F')}\\
&\le \xi(H_k)\\
&\le \max\{t'(i):1\le i\le n\}~~~~~~~~~~~~~~~~~~\text{(from Theorem~\ref{thm:maxsetstree})}\\
&=\left\{
\begin{array}{ll}
      \max\left\lbrace 2^{\frac{n-(2i+1)}{2}}+1:i=1,2,\ldots,n \right\rbrace &~~~~\text{if }n\text{ is odd} \\
      & \\
      \max\left\lbrace 2^{\frac{n-2i}{2}}+1:i=1,2,\ldots,n \right\rbrace &~~~~\text{if }n\text{ is even} \\
\end{array}\right.\\
&=\left\{
\begin{array}{ll}
      2^{\frac{n-3}{2}}+1 &~~~~\text{if }n\text{ is odd} \\
      & \\
      2^{\frac{n-2}{2}}+1 &~~~~\text{if }n\text{ is even}\\
\end{array}\right..
\end{align*}

\end{proof}

\begin{theorem}\label{thm:maxratioforest}
For $n\ge 1$,
\[ \maxmodf{n}\le \left\{
\begin{array}{ll}
      2^{\frac{n-1}{2}}+\frac{n-1}{2} & \textnormal{if $n$ is odd} \\
      & \\
      2^{\frac{n-2}{2}}+\frac{n}{2} & \textnormal{if $n$ is even}\\
\end{array} 
\right.. \]
\end{theorem}
\begin{proof}
As in the proof of Theorem~\ref{thm:graphrootbound}, we actually prove a stronger result, bounding the ratios of consecutive coefficients. The Enestr\"{o}m-Kakeya Theorem then applies to obtain the bound the roots. We proceed by induction on $n$. 

For $n=1,2,3,4$ the results hold by checking all forests of order at most $4$. Suppose the result holds for all $4\le k\le n-1$ and let $F$ be a forest on  $n$ vertices. Note that if $F=\overline{K_{n}}$, then the largest ratio of consecutive coefficients of $i(F,x)$ can easily be verified to be $n$ which is less than the result in either case, so suppose $F$ has at least one edge and therefore at least one leaf. Let $v$ be a leaf of $F$ and let $u$ be adjacent to $v$. Note that $\alpha(F-\{u,v\})\le\alpha(F-v)\le \alpha(F-\{u,v\})+1$ and for our argument, we assume that $\alpha(F-v)= \alpha(F-\{u,v\})$ as our arguments will hold (and be even shorter) when $\alpha(F-v)= \alpha(F-\{u,v\})+1$. To simplify notation, let $\alpha=\alpha(F-v)=\alpha(F-\{u,v\})$. By Proposition~\ref{prop:deletion}, we have
\begin{align*}
i(F,x)&=i(F-v,x)+x\cdot i(F-\{u,v\},x)\\
&=\sum_{k=0}^{\alpha}i_{k}^{F-v}x^k+x\sum_{k=0}^{\alpha}i_{k}^{F-\{u,v\}}x^k\\
&=1+\sum_{k=1}^{\alpha}\left(i_{k}^{F-v}+i_{k-1}^{F-\{u,v\}}\right)x^k+i_{\alpha}^{F-\{u,v\}}x^{\alpha+1}. \tag{2}\label{eq:2}
\end{align*}

We need to show that $\frac{i_{k}^F}{i_{k+1}^F}$ is bounded above by the desired value and from (\ref{eq:2}), we see that $\frac{i_{k}^F}{i_{k+1}^F}$ can take on the following forms, 

$$
\frac{1}{n}, \frac{i_{k}^{F-v}+i_{k-1}^{F-\{u,v\}}}{i_{k+1}^{F-v}+i_{k}^{F-\{u,v\}}}\text{ for $k=1,2,\ldots\alpha(F-\{u,v\})-1$, and } \frac{i_{\alpha}^{F-v}+i_{\alpha-1}^{F-\{u,v\}}}{i_{\alpha}^{F-\{u,v\}}}.
$$

The first ratio, $\tfrac{1}{n}$, clearly satisfies the desired bound regardless of the parity of $n$. We now only need to verify the remaining two forms of $\frac{i_{k}^F}{i_{k+1}^F}$. We will do this in two cases depending on the parity of $n$. \\

\noindent\textbf{Case 1:} $n$ is odd. \\

We apply the inductive hypothesis to get,
\begin{align*}
\frac{i_{k}^{F-v}+i_{k-1}^{F-\{u,v\}}}{i_{k+1}^{F-v}+i_{k}^{F-\{u,v\}}}&\le \frac{\left(2^{\frac{n-3}{2}}+\frac{n-1}{2}\right)i_{k+1}^{F-v}+\left(2^{\frac{n-3}{2}}+\frac{n-3}{2}\right)i_{k}^{F-\{u,v\}}}{i_{k+1}^{F-v}+i_{k}^{F-\{u,v\}}}\\
&\le \frac{\left(2^{\frac{n-3}{2}}+\frac{n-1}{2}\right)\left(i_{k+1}^{F-v}+i_{k}^{F-\{u,v\}}\right)}{i_{k+1}^{F-v}+i_{k}^{F-\{u,v\}}}\\
&=2^{\frac{n-3}{2}}+\tfrac{n-1}{2}\\
&< 2^{\frac{n-1}{2}}+\tfrac{n-1}{2}. 
\end{align*}

For the last ratio, we have,
\begin{align*}
\frac{i_{\alpha}^{F-v}+i_{\alpha-1}^{F-\{u,v\}}}{i_{\alpha}^{F-\{u,v\}}}&=\frac{i_{\alpha}^{F-v}}{i_{\alpha}^{F-\{u,v\}}}+\frac{i_{\alpha-1}^{F-\{u,v\}}}{i_{\alpha}^{F-\{u,v\}}}\\
&\le \frac{\xi(F-v)}{\xi(F-\{u,v\})}+2^{\frac{n-3}{2}}+\tfrac{n-3}{2} &\text{(by the inductive hypothesis)}\\
&\le 2^{\frac{n-3}{2}}+1+2^{\frac{n-3}{2}}+\tfrac{n-3}{2} &\text{(by the Lemma~\ref{lem:topratiobound})}\\
&=2^{\frac{n-1}{2}}+\tfrac{n-1}{2}.
\end{align*}
Therefore, the result holds when $n$ is odd by the Enestr\"{o}m-Kakeya Theorem.\\

\noindent\textbf{Case 2:} Suppose that $n$ is even. 

Then we apply the inductive hypothesis to get,
\begin{align*}
\frac{i_{k}^{F-v}+i_{k-1}^{F-\{u,v\}}}{i_{k+1}^{F-v}+i_{k}^{F-\{u,v\}}}&\le \frac{\left(2^{\frac{n-2}{2}}+\frac{n-2}{2}\right)i_{k+1}^{F-v}+\left(2^{\frac{n-4}{2}}+\frac{n-2}{2}\right)i_{k}^{F-\{u,v\}}}{i_{k+1}^{F-v}+i_{k}^{F-\{u,v\}}}\\
&\le \frac{\left(2^{\frac{n-2}{2}}+\frac{n-2}{2}\right)\left(i_{k+1}^{F-v}+i_{k}^{F-\{u,v\}}\right)}{i_{k+1}^{F-v}+i_{k}^{F-\{u,v\}}}\\
&=2^{\frac{n-2}{2}}+\tfrac{n-2}{2}\\
&< 2^{\frac{n-2}{2}}+\tfrac{n}{2}. 
\end{align*}

For the last ratio, we have,
\begin{align*}
\frac{i_{\alpha}^{F-v}+i_{\alpha-1}^{F-\{u,v\}}}{i_{\alpha}^{F-\{u,v\}}}&=\frac{i_{\alpha}^{F-v}}{i_{\alpha}^{F-\{u,v\}}}+\frac{i_{\alpha-1}^{F-\{u,v\}}}{i_{\alpha}^{F-\{u,v\}}}\\
&\le \frac{\xi(F-v)}{\xi(F-\{u,v\})}+2^{\frac{n-4}{2}}+\tfrac{n-2}{2} &\text{(by the inductive hypothesis)}\\
&\le 2^{\frac{n-4}{2}}+1+2^{\frac{n-4}{2}}+\tfrac{n-2}{2} &\text{(by the Lemma~\ref{lem:topratiobound})}\\
&=2^{\frac{n-2}{2}}+\tfrac{n}{2}.
\end{align*}
Therefore,the result holds when $n$ is even by the Enestr\"{o}m-Kakeya Theorem.
\end{proof}

\begin{corollary}\label{cor:maxtreerootupperbound}
For $n\ge 1$,
\[ \maxmodt{n}\le \left\{
\begin{array}{ll}
      2^{\frac{n-1}{2}}+\frac{n-1}{2} &~~~~\textnormal{if $n$ is odd} \\
      & \\
      2^{\frac{n-2}{2}}+\frac{n}{2} &~~~~\textnormal{if $n$ is even}\\
\end{array} 
\right.. \]\hfill\QED
\end{corollary}

We remark that, at least in terms of the bounds on the ratio of consecutive coefficients, this is best possible as there are forests that achieve these bounds. Let $n$ be odd, and consider the graph $T_{\frac{n-1}{2}}$ as previously defined and pictured in Figure~\ref{fig:Tk}. The independence polynomial of this tree has $2^{\frac{n-1}{2}}+\frac{n-1}{2}$ as its last ratio of consecutive coefficients. If $n$ is even then look at the forest $T_{\frac{n-2}{2}}\cup K_1$, whose independence polynomial has $2^{\frac{n-2}{2}}+\frac{n}{2}$ as its last ratio of consecutive coefficients.

We have shown that the bounds on the ratio of consecutive coefficients are tight, but are these bounds tight on the roots? It is not always the case that the upper bound on the moduli of the roots of a polynomial is tight, even for trees and forests. Take for example, the tree $K_{1,30}$ which has $30$ as an upper bound on the roots from Enestr\"{o}m-Kakeya but its actual root of largest modulus is approximately $2.023777128$. It gets even worse when we consider taking the disjoint union of $k$ copies of $K_{1,30}$. This forest will have the same root of maximum modulus but the bound on the root from Enestr\"{o}m-Kakeya is $30k$, which is unbounded. Fortunately, it turns out that the bound we found in Theorem~\ref{thm:maxratioforest} is asymptotically tight when $n$ is odd.

For the case where $n$ is even in the next proof we require the definition of the tree $T_k'$ as shown in Figure~\ref{fig:Tk'}. Let $T_k'$ be the graph obtained by adding two leaves to each vertex in $K_2$ and then gluing a leaf of the resulting graph to the central vertex of $T_k$.
\begin{figure}[htp]
\def\c{1}
\def\r{1}
\centering
\scalebox{\c}{
\begin{tikzpicture}
\begin{scope}[every node/.style={circle,thick,draw,fill}]

     \node[label=above:$e$] (12) at (3*\r,-1*\r) {};
     \node[label=above left:$d$] (11) at (3*\r,1*\r) {}; 
     \node[label=above left:$c$] (10) at (2*\r,0*\r) {};
     \node[label=above left:$b$] (9) at (1*\r,-1*\r) {};        
      \node[label=above left:$a$] (8) at (1*\r,0*\r) {};
  
     \node[label=above left:$y_{k}$] (7) at (-2*\r,-1.5*\r) {};
     \node[label=above left:$x_{k}$] (6) at (-1*\r,-1.5*\r) {}; 
     \node[label=above left:$y_2$] (5) at (-2*\r,0*\r) {};
     \node[label=above left:$x_2$] (4) at (-1*\r,0*\r) {};   
    
    \node[label=above left:$y_1$](3) at (-2*\r,1*\r) {};
    \node[label=above left:$x_1$] (2) at (-1*\r,1*\r) {};
    \node[label=above right:$z$] (1) at (0*\r,0*\r) {};
\end{scope}

\begin{scope}
    \path [-] (7) edge node {} (6);
    \path [-] (1) edge node {} (6);
    
    \path [-] (1) edge node {} (4);
    \path [-] (4) edge node {} (5);
    
    \path [-] (1) edge node {} (2);
    \path [-] (2) edge node {} (3);
    
    \path [-] (1) edge node {} (8);
    \path [-] (8) edge node {} (9);
    \path [-] (8) edge node {} (10);
    \path [-] (10) edge node {} (11);
    \path [-] (10) edge node {} (12); 
\end{scope}

\path (4) -- node[auto=false]{\textbf{\vdots}} (6);
\end{tikzpicture}}
\caption{The tree $T_k'$ on $2(k+3)$ vertices that has an independence root in $[-2^{k+1}-k-4,-2^k)$.}%
\label{fig:Tk'}%
\end{figure}
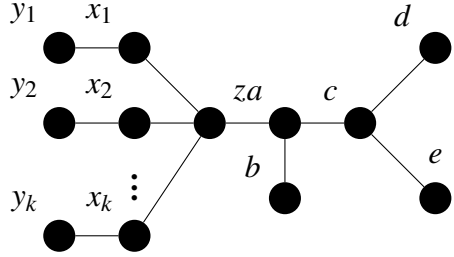

\begin{proposition}\label{prop:treerootslowerbound}
For all $n\ge 1$,
\[ \maxmodt{n}\ge \left\{
\begin{array}{ll}
      2^{\frac{n-1}{2}} &~~~~\textnormal{if } n \textnormal{ is odd} \\
      & \\
      2^{\frac{n-6}{2}} &~~~~\textnormal{if } n \textnormal{ is even} \\
\end{array} 
\right.. \]
\end{proposition}
\begin{proof}
The proof is similar to the proof of Proposition~\ref{prop:graphroots}, finding trees that have real independence roots of large modulus. \\

\noindent \textbf{Case 1:} $n$ is odd.\\

If $n=1$, then the result clearly holds so we may assume $n\ge 1$. Let $n=2k+1$ (note $k\ge 1$), so that $k=\frac{n-1}{2}$, and set $T=T_k$ as in Figure~\ref{fig:Tk}. A simple calculation via Proposition~\ref{prop:deletion} shows that $i(T,x)=(1+2x)^k+x(1+x)^k$. We will use the Intermediate Value Theorem to show that $i(T,x)$ has a real root to the left of $-2^k$. Now,

\begin{align*}
i(T,-2^k)&=(1-2^{k+1})^k-2^k(1-2^k)^k\\
&=(-1)^k\left((2^{k+1}-1)^k-(2^{k+1}-2)^k\right),
\end{align*}
so $i(T,-2^k)$ has the same sign as $(-1)^k$. On the other hand, $i(T,x)$ has sign $(-1)^{k+1}$ as $x$ tends to $\infty$We now compute the limit as $x$ tends to $-\infty$. Thus, $i(T_k,x)$ alternates sign on $(-\infty,-2^k]$, so by IVT it must have a real root in the interval $(-\infty,-2^k)$. We remark that from Theorem~\ref{thm:maxratioforest}, that $i(T,x)$ actually has a real root in the interval $[-2^k-k,-2^k)$.\\

\noindent\textbf{Case 2:} $n$ is even.

For $n=2$ and $4$, the result is clear. For $n\ge 6$, we will show that $T_k'$, the graph in Figure~\ref{fig:Tk'}, has a real root to the left of $-2^k$. Let $n=2(k+3)$ for $k\ge 0$. If $k=0$, then $i(T_k',x)=(1+x)^2(1+4x+x^2)$ which has roots $-1$, $-2+\sqrt{3}$, and $-2-\sqrt{3}$, with $-2-\sqrt{3}$ being to the left of $-2^{\frac{6-6}{2}}=-1)$. If $k=1$, then $i(T_k',x)={x}^{5}+9\,{x}^{4}+22\,{x}^{3}+21\,{x}^{2}+8\,x+1$, which has its largest root at approximately $-5.7833861$, which is to the left of $-2^{\frac{8-6}{2}}=-4$. Since the result holds for $k=0,1$, we may now assume that $k\ge 2$. 

Using Proposition~\ref{prop:deletion}, we find that $$i(T_k',x)=(1+2x)^k(1+5x+6x^2+2x^3)+x(1+x)^k(1+4x+4x^2+x^3).$$ Let $g(x)=1+5x+6x^2+2x^3$ and $h(x)=1+4x+4x^2+x^3$. We can easily verify that $g(x)<0$ for all $x\le -2$ and $h(x)<0$ for all $x\le -3$. Moreover, $h(x)=g(x)-x(x+1)^2$. We consider the function $$f(x)=(-2x-1)^k(1+5x+6x^2+2x^3)+x(-x-1)^k(1+4x+4x^2+x^3),$$ so that $i(T_k',x)=(-1)^kf(x)$. Now,

\begin{align*}
f(-2^k)&=(2^{k+1}-1)^kg(-2^k)-2^k(2^k-1)^kh(-2^k)\\
&=(2^{k+1}-1)^kg(-2^k)-2^k(2^k-1)^k(g(-2^k)+2^k(1-2^k)^2)\\
&=g(-2^k)((2^{k+1}-1)^k-(2^{k+1}-2)^k)-2^{2k}(2^k-1)^{k+2}
\end{align*}

\noindent and since $g(-2^k)$ and  $-2^{2k}(2^k-1)^{k+2}$ are both negative for $k\ge 2$, it follows that $f(-2^k)<~0$. Therefore, $i(T_k',x)$ has sign $(-1)^k(-1)=(-1)^{k+1}$. On the other hand, $i(T_k',x)$ has sign $(-1)^{k+4}=(-1)^{k}$ as $x$ tends to $\infty$. Thus, by the IVT, $i(T_k',x)$ has a real root to the left of $-2^k$. From, Theorem~\ref{thm:maxratioforest} and the Enestr\"{o}m-Kakeya Theorem, $i(T,x)$ has no roots in $(-\infty,-2^{k+1/2}-k-3)$, so $i(T_k',x)$ has a root in the interval $[-2^{k+1/2}-k-3,-2^k)$.

\end{proof}

Tables~\ref{tab:boundcomptreeodd} and \ref{tab:boundcomptreeeven} show values of $\maxmodt{n}$ for small values of $n$ in comparison to our bounds.

\begin{table}[h]
\parbox{.45\linewidth}{
\begin{tabular}{|c|c|c|c|}
\hline
$n$ & $2^{\frac{n-1}{2}}$ & $\maxmodt{n}$ & $2^{\frac{n-1}{2}}+\frac{n-1}{2}$\\
\hline
$3$ & $2$ & $2.61803398900000$ & $3$\\
\hline
$5$ & $4$ & $5.04891733952231$ & $6$\\
\hline
$7$ & $8$ & $9.49699733952714$ & $11$\\
\hline
$9$ & $16$ & $17.9705962347393$ & $20$\\
\hline
$11$ & $32$ & $34.4632033453548$ & $37$\\
\hline
$13$ & $64$ & $66.9662907779610$ & $70$\\
\hline
$15$ & $128$ & $131.473379027662$ & $135$\\
\hline
$17$ & $256$ & $259.980782682655$ & $264$\\
\hline
\end{tabular}
\caption{Comparing $\maxmodt{n}$ to our bounds for odd $n$ }\label{tab:boundcomptreeodd}
}
\hfill
\parbox{.45\linewidth}{
\begin{tabular}{|c|c|c|c|}
\hline
$n$ & $2^{\frac{n-6}{2}}$ & $\maxmodt{n}$ & $2^{\frac{n-2}{2}}+\frac{n}{2}$\\
\hline
$2$ & $0.25$ & $0.5$ & $2$\\
\hline
$4$ & $0.5$ & $1.77423195656734$ & $4$\\
\hline
$6$ & $1$ & $3.732050808$ & $7$\\
\hline
$8$ & $2$ & $5.78338611675281$ & $12$\\
\hline
$10$ & $4$ & $10.0833151322046$ & $21$\\
\hline
$12$ & $64$ & $18.5001015662614$ & $38$\\
\hline
$14$ & $8$ & $34.9710040067543$ & $71$\\
\hline
$16$ & $16$ & $67.4665144832128$ & $136$\\
\hline
\end{tabular}
\caption{Comparing $\maxmodt{n}$ to our bounds for even $n$ }\label{tab:boundcomptreeeven}}
\end{table}

Although the bounds on $\maxmodt{n}$ are not as tight for even $n$ as for odd $n$, Corollary~\ref{cor:maxtreerootupperbound} and Proposition~\ref{prop:treerootslowerbound} give the following corollary for all $n$.

\begin{corollary}
$$\frac{\log_2(\maxmodt{n})}{n}=\frac{1}{2}+o(1).$$ \hfill\QED
\end{corollary}


\section{Conclusion and Open Problems}\label{sec:conclusion}
We were able to prove upper and lower bounds on $\maxmod{n}$ and $\maxmodt{n}$ for all $n$ but questions remain about tightening our bounds for all graphs and about the growth rate of the moduli of independence roots for other families of graphs.

One highly structured and highly interesting family of graphs is the family of well-covered graphs \cite{Favaron1982,nowawcgirth,Plummer1970}, that is, graphs with all maximal independent sets of the same size. For each well-covered graph with independence number $\alpha$, it is known that all of its independence roots lie in the disk $|z|\le \alpha$ and there are well-covered graphs with independence roots arbitrarily close to the boundary \cite{BDN2000}. This difference between the independence roots of graphs and well-covered graphs begs the question of what happens for well-covered trees? Finbow et al. \cite{nowawcgirth} showed that every well-covered tree is obtained by attaching a leaf to every vertex of another tree. This construction of attaching a leaf to every vertex in a graph $G$ is know as the graph star operation, the resulting graph denoted $G^{\ast}$. Levit and Mandrescu \cite{Levit2008} proved a formula for $i(G^{\ast},x)$ in terms of $i(G,x)$ for all graphs $G$. Using Maple and nauty \cite{McKay2014}, we were able exploit this formula to verify that all well-covered trees on $n\le 40$ vertices have their independence roots contained in the unit disk! 

This makes it extremely tempting to conjecture that the independence roots of all well-covered trees are contained in the unit disk. However, the relationship between the independence roots of a tree and the independence roots of its well-covered extension are bound by the properties of M\"{o}bius transformations (this relationship was used by the authors in \cite{BrownCameron2018stability,BrownCameron2018vwc}). Drawing on the theory of these transfomrations, it can be shown that any tree $T$ with independence roots to the right of the line $\text{Re}(z)=\frac{1}{2}$, will yield a well-covered tree $T^{\ast}$ with independence roots outside of the unit disk. It was shown in \cite{BrownCameron2018stability}, that there are trees with independence roots arbitrarily far in the right half of $\mathbb{C}$, therefore, there are well-covered trees with independence roots outside of the unit disk. The tantalizing question remains:

\begin{question}
What is the maximum modulus of an independence root of a well-covered tree on $n$ vertices?
\end{question}

Our bounds on $\log_3(\maxmod{n})$ and $\log_2(\maxmodt{n})$ are very good asymptotically and a fairly good estimate for all $n$. Nevertheless, from computations with Maple and nauty \cite{McKay2014}, we have the following conjectures.

\begin{conjecture}
If $G$ is a graph on $n$ vertices, then for $n\ge 3$,
\[ \maxmod{n}\le \left\{
\begin{array}{ll}
      2\cdot 3^{\frac{n-3}{3}}+\frac{n}{3} &~~~~\text{if}\ n\equiv 0 \mod 3 \\      
      & \\
      3^{\frac{n-1}{3}}+\frac{n-1}{3} &~~~~\text{if}\ n\equiv 1 \mod 3 \\      
      & \\
      4\cdot 3^{\frac{n-5}{3}}+\frac{n+1}{3} &~~~~\text{if}\ n\equiv 2 \mod 3 \\      
\end{array} 
\right. \]
\end{conjecture}

\begin{conjecture}
The graphs $G_0,G_1,$ and $G_2$ are the only graphs to achieve $\maxmod{n}$.
\end{conjecture}

\begin{conjecture}
If $T$ is a tree on $n$ vertices with $n\ge 6$ even , then,
\[ \maxmodt{n}\le 2^{\frac{n-4}{2}}+\frac{n+2}{2} \]
\end{conjecture}

\begin{conjecture}
The trees $T_{\frac{n-1}{2}}$ and $T'_{\frac{n-6}{2}}$ (see Figures~\ref{fig:Tk}~and~\ref{fig:Tk'}) are the only trees to achieve $\maxmodt{n}$ for $n$ odd and even respectively.
\end{conjecture}

\vspace{0.25in}
\noindent {\bf Acknowledgements}
\vspace{0.1in}
J.I. Brown acknowledges support from NSERC (grant application RGPIN-2018-05227).

\bibliographystyle{spmpsci}      
\bibliography{Maxmodtrees}   

\begin{thebibliography}{10}
\providecommand{\url}[1]{{#1}}
\providecommand{\urlprefix}{URL }
\expandafter\ifx\csname urlstyle\endcsname\relax
  \providecommand{\doi}[1]{DOI~\discretionary{}{}{}#1}\else
  \providecommand{\doi}{DOI~\discretionary{}{}{}\begingroup
  \urlstyle{rm}\Url}\fi

\bibitem{BrownCameron2018stability}
Brown, J.I., Cameron, B.: {On the stability of independence polynomials}.
\newblock Electron. J. Combin. \textbf{25}(1) (2018)

\bibitem{BrownCameron2018vwc}
Brown, J.I., Cameron, B.: {On the unimodality of independence polynomials of
  very well-covered graphs}.
\newblock Discrete Math. \textbf{341}(4), 1138--1143 (2018)

\bibitem{BrownColbourn1992}
Brown, J.I., Colbourn, C.J.: {Roots of the reliability polynomial}.
\newblock SIAM J. Discrete Math. \textbf{5}(4), 571--585 (1992)

\bibitem{BDN2000}
Brown, J.I., Dilcher, K., Nowakowski, R.J.: {Roots of independence polynomials
  of well-covered graphs}.
\newblock J. Algebraic Combin. \textbf{11}, 197--210 (2000)

\bibitem{INDROOTS}
Brown, J.I., Hickman, C.A., Nowakowski, R.J.: {On the location of the roots of
  independence polynomials}.
\newblock J. Algebraic Combin. \textbf{19}, 273--282 (2004)

\bibitem{BrownMol2017}
Brown, J.I., Mol, L.: {On the roots of all-terminal reliability polynomials}.
\newblock Discrete Math. \textbf{340}(6), 1287--1299 (2017).
\newblock \doi{10.1016/j.disc.2017.01.024}.
\newblock \urlprefix\url{http://dx.doi.org/10.1016/j.disc.2017.01.024}

\bibitem{BrownNowakowski2001}
Brown, J.I., Nowakowski, R.J.: {Bounding the roots of independence
  polynomials}.
\newblock Ars Combin. \textbf{58}, 113--120 (2001)

\bibitem{BrownTufts2014}
Brown, J.I., Tufts, J.: {On the Roots of Domination Polynomials}.
\newblock Graphs Combin. \textbf{30}, 527--547 (2014)

\bibitem{Chudnovsky2007}
Chudnovsky, M., Seymour, P.: {The roots of the independence polynomial of a
  clawfree graph}.
\newblock J. Combin. Theory Ser. B \textbf{97}(3), 350--357 (2007)

\bibitem{Csikvari2013}
Csikv{\'{a}}ri, P.: {Note on the smallest root of the independence polynomial}.
\newblock Combin. Probab. Comput. \textbf{22}(1), 1--8 (2013)

\bibitem{Csikvari2011}
Csikv{\'{a}}ri, P., Oboudi, M.R.: {On the roots of edge cover polynomials of
  graphs}.
\newblock European J. Combin. \textbf{32}(8), 1407--1416 (2011)

\bibitem{Enestrom}
Enestr{\"{o}}m, G.: {H{\"{a}}rledning af en allm{\"{a}}n formel f{\"{o}}r
  antalet pension{\"{a}}rer, som vid en godtyeklig tidpunkt f{\"{o}}refinnas
  inom en sluten pensionslcassa}.
\newblock {\"{O}}fversigt af Vetenskaps-Akademiens F{\"{o}}rhandlingar
  \textbf{50}, 405--415 (1893)

\bibitem{Favaron1982}
Favaron, O.: {Very well covered graphs}.
\newblock Discrete Math. \textbf{42}, 177--187 (1982)

\bibitem{nowawcgirth}
Finbow, A., Hartnell, B., Nowakowski, R.J.: {A characterization of well covered
  graphs of girth 5 or greater}.
\newblock J. Combin. Theory Ser. B \textbf{57}, 44--68 (1993)

\bibitem{INDFIRST}
Gutman, I., Harary, F.: {Generalizations of the matching polynomial}.
\newblock Util. Math. \textbf{24}, 97--106 (1983)

\bibitem{Jou2000}
Jou, M.J., Chang, G.J.: {The number of maximum independent sets in graphs}.
\newblock Taiwanese J. Math. \textbf{4}(4), 685--695 (2000)

\bibitem{Kakeya1912}
Kakeya, S.: {On the limits of the roots of an algebraic equation with positive
  coefficients}.
\newblock Tohoku Math. J. (2) pp. 140--142 (1912)

\bibitem{INDPOLY}
Levit, V.E., Mandrescu, E.: {The independence polynomial of a graph-a survey}.
\newblock In: Proceedings of the 1st International Conference on Algebraic
  Informatics, vol.~V, pp. 233--254. Aristotle Univ. Thessaloniki, Thessaloniki
  (2005)

\bibitem{Levit2008}
Levit, V.E., Mandrescu, E.: {On the roots of independence polynomials of almost
  all very well-covered graphs}.
\newblock Discrete Appl. Math. \textbf{156}, 478--491 (2008)

\bibitem{Makowsky2014}
Makowsky, J.A., Ravve, E.V., Blanchard, N.K.: {On the location of roots of
  graph polynomials}.
\newblock European J. Combin. \textbf{41}, 1--19 (2014)

\bibitem{McKay2014}
McKay, B.D., Piperno, A.: {Practical graph isomorphism II}.
\newblock J. Symbolic Comput. \textbf{60}, 94--112 (2014)

\bibitem{MoonMoser1965}
Moon, J.W., Moser, L.: {On cliques in isoregular graphs.}
\newblock Israel J. Math. \textbf{3}, 23--28 (1965)

\bibitem{Oboudi2018treeroots}
Oboudi, M.R.: {On the largest real root of independence polynomials of trees}.
\newblock Ars Combin. \textbf{137}, 149--164 (2018)

\bibitem{Plummer1970}
Plummer, M.D.: {Some Covering Concepts in Graphs}.
\newblock J. Combin.Theory \textbf{8}, 91--98 (1970)

\bibitem{RoyleSokal2004}
Royle, G., Sokal, A.D.: {The Brown-Colbourn conjecture on zeros of reliability
  polynomials is false}.
\newblock J. Combin. Theory Ser. B \textbf{91}(2), 345--360 (2004)

\bibitem{Sagan1988}
Sagan, B.E.: {A note on independent sets in trees}.
\newblock SIAM J. Discrete Math. \textbf{1}(1), 105--108 (1988)

\bibitem{Sokal2001}
Sokal, A.D.: {Bounds on the Complex Zeros of (Di)Chromatic Polynomials and
  Potts-Model Partition Functions}.
\newblock Combin. Probab. Comput. \textbf{10}, 41--77 (2001)

\bibitem{Sokal2004}
Sokal, A.D.: {Chromatic Roots are Dense in the Whole Complex Plane}.
\newblock Combin. Probab. Comput. \textbf{2}, 221--261 (2004)

\bibitem{Wilf1986}
Wilf, H.S.: {The Number of Maximal Independent Sets in a Tree}.
\newblock SIAM J. Alg. Disc. Meth. \textbf{7}(1), 125--130 (1986)

\end{thebibliography}

\end{document}